\documentclass{amsart}

\usepackage{amsmath,amssymb,amsthm}
\usepackage{hyperref}
\usepackage{xcolor}
\usepackage{bbm}

\usepackage[small,nohug,heads=vee,dpi=600]{diagrams}
\diagramstyle[labelstyle=\scriptstyle]
\newarrow {Dashto} {}{dash}{}{dash}>
\newarrow {Dash} {}{dash}{}{dash}{}
\newarrow {Doubleto} {}{=}{}{=}{=>}

\newtheorem{prop}{Proposition}[section]
\newtheorem{thm}[prop]{Theorem}

\newtheorem{cor}[prop]{Corollary}

\newtheorem{defn}[prop]{Definition}

\theoremstyle{definition}

\newtheorem{rem}[prop]{Remark}
\newtheorem{exwith}[prop]{Example}

\newtheorem*{ack}{Acknowledgement}

\def\co{\colon\thinspace}

\newcommand{\BB}{\mathcal B}

\newcommand{\C}{\mathbb C}

\newcommand{\rmd}{\mathrm d}
\newcommand{\D}{\mathbb D}

\newcommand{\rme}{\mathrm e}

\newcommand{\HH}{\mathcal H}
\newcommand{\bfh}{\mathbf h}

\newcommand{\rmi}{\mathrm i}

\newcommand{\N}{\mathbb N}

\newcommand{\bfp}{\mathbf p}

\newcommand{\bfq}{\mathbf q}

\newcommand{\R}{\mathbb R}
\newcommand{\RR}{\mathcal R}

\newcommand{\bfs}{\mathbf s}

\newcommand{\bft}{\mathbf t}

\newcommand{\VV}{\mathcal V}
\newcommand{\bfv}{\mathbf v}

\newcommand{\WW}{\mathcal W}
\newcommand{\bfw}{\mathbf w}

\newcommand{\bfx}{\mathbf x}
\newcommand{\bfy}{\mathbf y}

\newcommand{\Z}{\mathbb Z}

\newcommand{\bfz}{\mathbf z}

\newcommand{\lra}{\longrightarrow}
\newcommand{\ra}{\rightarrow}

\DeclareMathOperator{\ev}{\mathrm{ev}}

\DeclareMathOperator{\im}{\mathrm{Im}}
\DeclareMathOperator{\ind}{\mathrm{ind}}

\DeclareMathOperator{\Int}{\mathrm{Int}}

\DeclareMathOperator{\loc}{\mathrm{loc}}

\begin{document}

\author{Myeonggi Kwon}
\author{Kevin Wiegand}
\author{Kai Zehmisch}
\address{Mathematisches Institut, Justus-Liebig-Universit\"at Gie{\ss}en,
Arndtstra{\ss}e 2, D-35392 Gie{\ss}en, Germany}
\email{Myeonggi.Kwon@math.uni-giessen.de, Kai.Zehmisch@math.uni-giessen.de}
\address{Mathematisches Institut, Ruprecht-Karls-Universit\"at Heidelberg,
Im Neuenheimer Feld 205, D-69120 Heidelberg, Germany}
\email{kwiegand@mathi.uni-heidelberg.de}

\title[Diffeomorphism type via aperiodicity]
{Diffeomorphism type via aperiodicity in Reeb dynamics}

\date{09.10.2019}

\begin{abstract}
  We characterise boundary shaped
  disc like neighbourhoods
  of certain isotropic submanifolds
  in terms of aperiodicity of Reeb flows.
  We prove uniqueness of homotopy
  and diffeomorphism type
  of such contact manifolds
  assuming non-existence of short periodic Reeb orbits.
\end{abstract}

\thanks{This research is part of a project in the SFB/TRR 191
{\it Symplectic Structures in Geometry, Algebra and Dynamics},
funded by the DFG}

\maketitle

%%%%%%%%%%%%%%%%%%%%%%%%%%%%%%%%%%%%%%%%%%%%%%%%%%%%%%%%%%%%%%%%%%%%%%

\section{Introduction\label{sec:intro}}

In their seminal work
Gromov \cite{grom85} and Eliashberg \cite{elia90} observed
that foliations by holomorphic curves can be used
to prove uniqueness of the diffeomorphism
(in fact symplectomorphism)
type of minimal symplectic fillings
of the standard contact $3$-sphere,
i.e.\
all such fillings are diffeomorphic to the $4$-ball $D^4$.
The method they used,
the so-called {\it filling by holomorphic curves} method,
is obstructed by bubbling off of holomorphic spheres.
Related classification results in dimension $4$
can be found in
\cite{mcdu90,hind00,stip02,wend10,lmy17,siva17}.

On the other hand
Hofer \cite{hof93} discovered
a fundamental property of holomorphic curves in symplectisations;
non-compactness properties of holomorphic curves of finite energy
are strongly related to the existence of periodic Reeb orbits.
Combining the method of
filling by holomorphic curves
with the theory of finite energy planes
Eliashberg--Hofer \cite{eh94}
determined the diffeomorphism
(in fact contactomorphism)
type of certain contact manifolds with boundary $S^2$:
Any compact contact manifold
with boundary $S^2=\partial D^3$
is diffeomorphic to $D^3$
provided there exists a contact form
that is equal to the standard contact form on $D^3$
near the boundary $S^2$
such that the corresponding Reeb vector field
does not admit a periodic orbit
with period less than $\pi$.
A similar characterisation of $D^2\times S^1$
in terms of Reeb dynamics
was obtained by
Kegel--Schneider--Zehmisch \cite{kschz}.

In higher dimensions
the diffeomorphism type of symplectically aspherical fillings
of the standard contact sphere
was determined by Eliashberg--Floer--McDuff
\cite[Theorem~1.5]{mcdu91}:
Any such filling is diffeomorphic to the ball $D^{2n}$.
The proof they used
was refined to the so-called
{\it degree method}
(see Section \ref{subsec:filling} for an explanation)
by Barth--Geiges--Zehmisch \cite{bgz19}
allowing a much wider class of contact type boundaries,
see also \cite{bgm,gkz,kwzeh}.

The contact theoretic counterpart
in higher dimensions
was not clear for a while.
It was conjectured by
Bramham--Hofer \cite{bh12}
that the existence of trapped Reeb orbits
on a compact contact manifold,
whose boundary neighbourhoods
look like neighbourhoods
of $S^{2n}=\partial D^{2n+1}$ in $D^{2n+1}$,
implies the existence of periodic Reeb orbits.
A counterexample to that conjecture was given by
Geiges--R\"ottgen--Zehmisch \cite{grz14}.
It suggests
that the diffeomorphism type
in higher dimensional contact geometry
should be determined via a method not based
on non-existence of trapped orbits as done in 
Eliashberg--Hofer \cite{eh94}.

In fact,
using the degree method, Geiges--Zehmisch \cite{gz16b}
proved that any compact strict contact manifold
that has an aperiodic Reeb flow
is diffeomorphic to $D^{2n+1}$
provided that the following condition is satisfied:
A neighbourhood of the boundary
admits a strict contact embedding
into the standard $D^{2n+1}$ mapping the boundary
to $S^{2n}=\partial D^{2n+1}$.
This was generalised by
Barth--Schneider--Zehmisch \cite{bschz19}
to situations in which $D^{2n+1}$ is replaced by
the disc bundle of 
$\R\times T^*T^d\times\C\times\C^{n-1-d}$
whenever $n-1\geq d$.

The aim of this work
is to replace the torus $T^d$
by more general $d$-dimensional manifolds,
see Theorem \ref{thm:mainthm} below.
Again the argument will be based
on the construction of a proper degree $1$
evaluation map on the moduli space
of $1$-marked holomorphic discs
with varying Lagrangian boundary conditions.
The restriction to $T^d$ in \cite{bschz19}
was caused by the choice of the boundary conditions
set up for the holomorphic discs.
This led to trivialising the cotangent bundle
of $T^d$ in a Stein holomorphic fashion.
In order to replace $D^{2n+1}$ by
the disc bundle of 
$\R\times T^*Q\times\C\times\C^{n-1-d}$
for a wider class of manifolds $Q$
we choose different boundary conditions
for the holomorphic discs.
Instead of taking a foliation of $T^*Q$ by sections
we consider the foliation $T^*Q$ given by the cotangent fibres.
This will result in a more advanced analysis
for the holomorphic discs.
The essential point here
will be a target rescaling argument
in Section \ref{sec:compactness},
which was invented by Bae--Wiegand--Zehmisch \cite{bwz}
in the context of virtually contact structures,
to ensure $C^0$-bounds on holomorphic discs
in the situation of general manifolds $Q$.
Furthermore
in order to obtain $C^0$-bounds of holomorphic discs
along their boundaries in $T^*Q$-direction
we develop an integrated maximum principle
in Sections \ref{sec:potentialsont*q}
and \ref{subsec:inmaxprinc}.

%%%%%%%%%%%%%%%%%%%%%%%%%%%%%%%%%%%%%%%%%%%%%%%%%%%%%%%%%%%%%%%%%%%%%%

\section{Aperiodicity and boundary shape\label{sec:aperboundshape}}

Strict contact manifolds $(M,\alpha)$
are naturally equipped
with a nowhere vanishing vector field,
namely the Reeb vector field of $\alpha$.
Assuming $\alpha$ to be {\bf aperiodic},
i.e.\ assuming that the Reeb vector field
does not admit any periodic solution,
the diffeomorphism type of $M$
can be determined in many situations.
Here we are interested in comparing
compact manifolds with boundary $M$
with neighbourhoods of isotropic submanifolds
of the sort
$D\big(T^*Q\oplus\underline{\R}^{2n+1-2d}\big)$.
This requires boundary conditions
for the Reeb vector field as we will explain in the following:

%%%%%%%%%%%%%%%%%%%%%%%%%%%%%%%%%%%%%%%%%%%%%%%%%%%%%

\subsection{A model\label{subsec:amodel}}

Let $Q$ be a closed, connected Riemannian manifold
of dimension $d$ and let $n\in\N$ such that $n-1\geq d$.
Define a strict contact manifold $(C,\alpha_0)$
by setting
\[
C:=\R\times T^*Q\times\C\times\C^{n-1-d}
\]
and
\[
\alpha_0:=
\rmd b
+\lambda
+\frac12\big(x_0\rmd y_0-y_0\rmd x_0\big)
-\sum_{j=1}^{n-1-d}y_j\rmd x_j\;,
\]
where $b\in\R$,
$\lambda$ is the Liouville $1$-form of $T^*Q$,
$x_0+\rmi y_0$ and $x_j+\rmi y_j$
are coordinates on $\C$
and $\C^{n-1-d}$, resp.
Throughout the text
we will use vector notation $\bfx$ and $\bfy$
for the coordinate tuples
$(x_1,\ldots,x_{n-1-d})$ and
$(y_1,\ldots,y_{n-1-d})$, resp.,
so that we can abbreviate
\[
-\bfy\rmd\bfx=
-\sum_{j=1}^{n-1-d}y_j\rmd x_j\;.
\]
The Reeb vector field of $\alpha$ is given by $\partial_b$,
which is tangent to the real lines $\R\times\{*\}$.

By \cite[Theorem 6.2.2]{gei08}
$(C,\alpha_0)$ is the model neighbourhood
of an isotropic submanifold $Q$
of a strict contact manifold
provided that $Q$
has trivial symplectic normal bundle
and the dimension $d$ of $Q$ is smaller than $n$.
Observe,
that $(C,\alpha_0)$ is the contactisation
of the Liouville manifold
\[
\Big(
T^*Q\times\C\times\C^{n-1-d},\,
\lambda
+\frac12\big(x_0\rmd y_0-y_0\rmd x_0\big)
-\bfy\rmd\bfx
\Big)
\;.
\]
The statements about the model neighbourhood situation
and contactisation of course hold in the critical case $d=n$ also.
Simply ignore the Euclidean factors in the formulations.

%%%%%%%%%%%%%%%%%%%%%%%%%%%%%%%%%%%%%%%%%%%%%%%%%%%%%

\subsection{Fibrewise shaped\label{subsec:fibreshape}}

The space $C$ itself is the total space
of the stabilised cotangent bundle
$T^*Q\oplus\underline{\R}^{2n+1-2d}$.
Let $S\subset C$ be a hypersurface diffeomorphic to
the unit sphere bundle
$S\big(T^*Q\oplus\underline{\R}^{2n+1-2d}\big)$
such that
  \begin{enumerate}
    \item
      $S$ intersects each fibre transversely
      in a sphere
        \[
        S_q:=S\cap\big(T_q^*Q\oplus\R^{2n+1-2d}\big)\;,
        \quad q\in Q\;,
        \]
      of dimension $2n-d$, and
    \item
      each $S_q$ intersects the flow lines
      of $\partial_b$ in at most two points.
      We require transverse intersections
      if such a flow line intersects $S_q$
      in two points.
      Points of tangency,
      i.e.\ points that correspond to single intersections,
      form a submanifold diffeomorphic to a
      $(2n-d-1)$-sphere.
  \end{enumerate}
In view of condition (1)
we remark that the hypersurface $S$
bounds a bounded domain $D$ inside $C$,
whose closure is diffeomorphic
to the closed unit disc bundle
$D\big(T^*Q\oplus\underline{\R}^{2n+1-2d}\big)$.
Condition (2) will play an important role in
Section \ref{subsec:comviagluing}.
We call $S$ a {\bf shape}.

%%%%%%%%%%%%%%%%%%%%%%%%%%%%%%%%%%%%%%%%%%%%%%%%%%%%%

\subsection{Standard near the boundary}
\label{subsec:standnearbound}

Let $(M,\alpha)$ be a strict contact manifold
of dimension $2n+1$ that is
{\bf standard near the boundary},
i.e.\
  \begin{enumerate}
    \item
      connected, compact with boundary $\partial M$
      diffeomorphic to
      \[
      \partial M\cong S\big(T^*Q\oplus\underline{\R}^{2n+1-2d}\big)
      \]
    \item
      such that there exist
      an open collar neighbourhood $U\subset M$ of $\partial M$
      and an embedding
      $\varphi:(U,\partial U=\partial M)\ra (D,S)$
      such that $\varphi^* \alpha_0 = \alpha$ on $U$.
  \end{enumerate}
If $\varphi$ is given we will call $S$ the {\bf shape} of $M$.

In order to quantify
aperiodicity of $(M,\alpha)$
we denote by $\inf_0(\alpha)>0$
the minimal action of all {\it contractible} closed Reeb orbits
w.r.t.\ $\alpha$.
By Darboux's theorem $\inf_0(\alpha)$
is indeed positive.
For aperiodic $\alpha$
we set $\inf_0(\alpha)$ to be $\infty$.

A second ingredient for quantisation comes with
the subset
\[
Z:=\R\times T^*Q \times\D\times\C^{n-1-d}
\]
of $C$ denoting the closed unit disc in $\C$ by $\D$.
We may assume that $S\subset\Int Z$
by scaling radially via
$\big(t^2b,t^2\bfw,tz_0,t\bfz\big)$,
$t\in(0,1)$, if necessary.
The contact form $\alpha$ on $M$
will be replaced by $t^2\alpha$ accordingly.

%%%%%%%%%%%%%%%%%%%%%%%%%%%%%%%%%%%%%%%%%%%%%%%%%%%%%

\subsection{Main theorem}\label{subsec:mainthm}

We compare the homology,
homotopy and diffeomorphism type of $M$
with the one of
$D\big(T^*Q\oplus\underline{\R}^{2n+1-2d}\big)$.
This will be done
in terms of embeddings
\[
D\big(T^*Q\oplus\underline{\R}^{2n+1-2d}\big)
\lra M
\]
determined by a small neighbourhood of a section
$Q\ra S$ as constructed e.g.\ 
at the beginning of Section \ref{sec:homotopytype}.
We denote the image of such an embedding by
\[
M_0:=
D\big(T^*Q\oplus\underline{\R}^{2n+1-2d}\big)
\;.
\]

\begin{thm}
\label{thm:mainthm}
 Let $Q$ be an oriented, closed, connected Riemannian manifold
 of dimension $d$.
 Let $n\in\N$ such that $n-1\geq d$.
 Let $(M,\alpha)$ be a strict contact manifold
 that is standard near the boundary
 as described in Section \ref{subsec:standnearbound}.
 Assume that the shape $S\cong\partial M$ of $M$
 is contained in the interior of $(Z,\alpha_0)$.
 If $\inf_0(\alpha)\geq\pi$, then the following is true:
   \begin{enumerate}
   \item [(i)]
   Any embedding
   $Q\ra M$ given by a section $Q\ra S$
   induces isomorphisms of homology
   and surjections of fundamental groups.
   If in addition $\pi_1Q$ is abelian,
   then the surjections are injective.
   \item [(ii)]
   Assume that $\pi_1Q$ is abelian and
   that at least one
   of the following conditions is satisfied:
     \begin{enumerate}
     \item [(a)]
     $\pi_1Q$ is finite.
     \item [(b)]
     $Q$ is aspherical.
     \item [(c)]
     $Q$ is simple and $S\ra Q$ a trivial sphere bundle, or,
     more generally, $S$ is a simple space.
     \end{enumerate}
   Then $M$ is homotopy equivalent to $M_0$.
   \item [(iii)]
    If in addition to the assumptions in (ii)
    (including choices of one of the conditions (a)-(c))
    we have that $2n+1\geq 7$ and
    that the Whitehead group of $\pi_1Q$ is trivial,
    then $M$ is diffeomorphic to $M_0$.
   \end{enumerate}
\end{thm}

%%%%%%%%%%%%%%%%%%%%%%%%%%%%%%%%%%%%%%%%%%%%%%%%%%%%%

\subsection{Comments on Theorem \ref{thm:mainthm}}\label{subsec:comonthm}

In view of the contact connected sum
the bound $\pi$ in the theorem is optimal,
cf.\ \cite[Remark 1.3.(1)]{gz16b}.
The shape boundary condition
can be isotoped to a round shape
through shaped hypersurfaces.
Hence,
we recover the results
from \cite{gz16b,bschz19}
and obtain independence of the choice of metric.

The orientation of $Q$ will not be used
in the compactness argument below.
But will be needed for an orientation
of the moduli space.
Without orientation we only can talk about the mod-$2$
degree of the evaluation map.
Hence,
if $Q$ is not orientable
only part (i) of the theorem remains true
replacing homology by homology with
$\Z_2$-coefficients.

Similarly,
the boundary of $M$ is necessarily connected,
cf.\ \cite[Remark 1.3.(4)]{gz16b}.
Indeed,
suppose $\partial M$ has several components
that have individually a shape embedding
into potentially different stabilised cotangent bundles.
Here different $Q$'s with varying dimensions are allowed.
$M$ itself satisfies the remaining stated properties
from Theorem \ref{thm:mainthm}.
In this situation
one can set up the moduli space of holomorphic discs
with respect to one distinguished boundary component;
the other components will come with the maximum principle
for holomorphic curves.
In other words
the holomorphic disc analysis will be uneffected
and the evaluation map on the moduli space will be of degree one.
This contradicts the fact
that no holomorphic disc can exceed
one of the additional boundary components
due to the maximum principle.

\begin{exwith}
\label{ex:asphericalthm}
 In view of the Hadamard--Cartan and the Farrell--Jones
 theorems the assumptions
 of Theorem \ref{thm:mainthm} part (b) in (ii) and (iii)
 are satisfied for all
 Riemannian manifolds $Q$ with abelian fundamental group
 and non-positive sectional curvature.
 Hence, we recover $Q=T^d$ from \cite{bschz19}.
\end{exwith}

\begin{exwith}
\label{ex:unitaryspherethm}
 A particular class of manifolds $Q$
 that satisfy the assumptions
 of Theorem \ref{thm:mainthm} part (c) in (ii) and (iii)
 are products of unitary groups and spheres
 of any dimensions.
 Indeed,
 such $Q$ always have stably trivial tangent bundle,
 are simple with fundamental group free abelian
 so that in particular the Whitehead group of those is trivial.
\end{exwith}

\begin{rem}
\label{rem:handlebody}
 If we know more about the handle body structure of $M$
 conditions on the topology of $Q$ can be relaxed.
 For example
 if $M$ has the homotopy type
 of a CW complex of codimension $2$
 so that the inclusion $\partial M\subset M$ is $\pi_1$-injective
 the assumption $\pi_1Q$ abelian in
 Theorem \ref{thm:mainthm} can be dropped everywhere.
 
 If $M$ admits a handle body structure
 with all handles of index at most $\ell$
 and if $d+\max(d,\ell)\leq2n-1$,
 then $M$ and $M_0$ are homotopy equivalent
 without any further conditions.
 This follows with the argument
 from \cite[Theorem 7.2]{bgz19}
 using the diagram in Section \ref{subsec:acobordism}.
 In particular,
 the CW-dimension of $M$ must be equal to $d$.
 In fact,
 one can conclude with the diffeomorphism type
 as in \cite[Theorem 9.4]{bgz19},
 cf.\ \cite[Example 9.5]{bgz19}.
\end{rem}

%%%%%%%%%%%%%%%%%%%%%%%%%%%%%%%%%%%%%%%%%%%%%%%%%%%%%%%%%%%%%%%%%%%%%%

\section{The degree method\label{sec:mainidea}}

We will explain the main idea
of the proof of Theorem \ref{thm:mainthm},
which will be given in
Sections \ref{sec:stholdiscs} -- \ref{sec:homotopytype}.

%%%%%%%%%%%%%%%%%%%%%%%%%%%%%%%%%%%%%%%%%%%%%%%%%%%%%

\subsection{Completion via gluing\label{subsec:comviagluing}}

Assuming $S\subset\Int Z$
we define smooth manifolds
\[
\hat{C}:=(C\setminus\Int D)\cup_{\varphi}M\;,
\qquad
\hat{Z}:=(Z\setminus\Int D)\cup_{\varphi}M
\]
by gluing via $\varphi$
and equip both with the contact form
\[
\hat{\alpha}:=\alpha_0\cup_{\varphi}\alpha
\]
that coincides with $\alpha$ on $M$
and with $\alpha_0$ on $C\setminus\Int D$.
Because of the contact embedding
$\varphi$ of $U\supset\partial M$
into $(Z,\alpha_0)$ this is well defined. 
According to the second shape condition
in Section \ref{subsec:fibreshape}
the gluing does not create
additional periodic Reeb orbits
inside $(\hat{C},\hat{\alpha})$
so that $\inf_0(\alpha)$ and $\inf_0(\hat{\alpha})$
coincide.

%%%%%%%%%%%%%%%%%%%%%%%%%%%%%%%%%%%%%%%%%%%%%%%%%%%%%

\subsection{Filling by holomorphic discs\label{subsec:filling}}

In order to prove Theorem \ref{thm:mainthm}
we will argue as in \cite{gz16b,bschz19}:
The Liouville manifold
\[
\Big(
T^*Q\times\D\times\C^{n-1-d},\,
\lambda
+\frac12\big(x_0\rmd y_0-y_0\rmd x_0\big)
-\bfy\rmd\bfx
\Big)
\]
is foliated by holomorphic discs
$\{\bfw\}\times\D\times\{\bfs+\rmi\bft\}$.
Using the Niederkr\"uger transformation
from Section \ref{subsec:symplectisation}
these discs can be lifted to
holomorphic discs in the
symplectisation of the contactisation $(Z,\alpha_0)$
and are called {\bf standard discs}. 
After gluing some of the standard discs will survive,
namely those which correspond to the end of $(\hat{Z},\hat{\alpha})$
in the symplectisation $(W,\omega)$ of $(\hat{Z},\hat{\alpha})$.
We will study the corresponding moduli space $\WW$
of holomorphic discs
\[
u=(a,f)\co\D\lra W
\]
subject to varying Lagrangian boundary conditions,
which will differ dramatically
from those used in \cite{gz16b,bschz19}.
The novelty lies in a new $C^0$-bound argument,
which allows a wider class of base manifolds $Q$.

It will turn out that the evaluation map
\[
 \begin{array}{rccc}
  \ev\co & \WW\times\D     & \lra & \hat{Z}\\
            & \bigl( (a,f),z\bigr) & \longmapsto     & f(z)
 \end{array}
\]
either is proper of degree one
or there will be breaking off of finite energy planes.
The first alternative allows
conclusions on the diffeomorphism type of $M$
with the $s$-cobordism theorem
as in \cite{bgz19}.
The second results in the existence of a {\bf short}
contractible periodic Reeb orbit of $\alpha$ on $M$
by a result of Hofer \cite{hof93}.
Short here means
that the action of the Reeb orbit
is bounded by the area of $\D$.

The condition $\inf_0(\alpha)\geq\pi$
will exclude breaking of holomorphic discs
along periodic Reeb orbits
of action less than $\pi$.
But in fact,
under the assumptions of
Theorem \ref{thm:mainthm}
the shape $S$ of $M$
actually is contained in
$\R\times T^*Q\times B_r(0)\times\C^{n-1-d}$
for $r\in(0,1)$.
Working out the proof of Theorem \ref{thm:mainthm}
with that slightly smaller radius $r$
we will see that requiring non-existence
of short periodic Reeb orbits with period bounded by $\pi r^2$ 
will be sufficient. 
In other words,
we can assume that $\inf_0(\alpha)>\pi r^2$
in order to prove properness of the evaluation map $\ev$.
To simplify notation we will assume $r=1$,
i.e.\ from now on we assume $\inf_0(\alpha)>\pi$.

%%%%%%%%%%%%%%%%%%%%%%%%%%%%%%%%%%%%%%%%%%%%%%%%%%%%%%%%%%%%%%%%%%%%%%

\section{Standard holomorphic discs\label{sec:stholdiscs}}

In this section we construct standard holomorphic discs.
We will follow \cite[Section 2]{gz16b} and \cite[Section 2]{bschz19}
adding adjustments to the current situation.

%%%%%%%%%%%%%%%%%%%%%%%%%%%%%%%%%%%%%%%%%%%%%%%%%%%%%

\subsection{The contactisation\label{subsec:contactisation}}

We consider the Liouville manifold
\[
  (V,\lambda_V):=
  \Big(
  T^*Q\times\D\times\C^{n-1-d},\,
  \lambda
  +\frac12\big(x_0\rmd y_0-y_0\rmd x_0\big)
  -\bfy\rmd\bfx
  \Big)
  \;,
\]
whose contactisation
$(\R\times V,\rmd b+\lambda_V)$
is $(Z,\alpha_0)$.
The induced contact structure $\xi_0=\ker\alpha_0$ on $Z$
is spanned by tangent vectors of the form
$v-\lambda_V(v)\partial_b$, $v\in TV$.

%%%%%%%%%%%%%%%%%%%%%%%%%%%%%%%%%%%%%%%%%%%%%%%%%%%%%

\subsection{Liouville manifold and potential\label{subsec:liouandpot}}

Denote by $J_{T^*Q}$
the almost complex structure on $T^*Q$
that is compatible with $\rmd\lambda$
and satisfies
$\lambda=-\rmd F\circ J_{T^*Q}$.
Here $F$ is a strictly plurisubharmonic potential
in the sense of \cite[Section 3.1]{gz12}
that coincides with the kinetic energy function
near the zero section of $T^*Q$
and interpolates to the length function
on the complement of a certain disc bundle in $T^*Q$,
see \cite[Section 3.1]{meop}.
In Section \ref{sec:potentialsont*q}
we will present a construction of $(F,J_{T^*Q})$.

Define an almost complex structure
on the Liouville manifold $(V,\lambda_V)$
by setting
\[
J_V:=J_{T^*Q}\oplus\rmi\oplus\rmi
\;.
\]
$J_V$ is compatible with the symplectic form $\rmd\lambda_V$
and satisfies $\lambda_V=-\rmd\psi\circ J_V$,
where $\psi$ is the strictly plurisubharmonic potential
\[
\psi(\bfw,z_0,\bfz):=
F(\bfw)+
\frac14|z_0|^2+
\frac12|\bfy|^2
\]
denoting by $\bfw\in T^*Q$ a co-vector of $Q$,
$z_0\in\D$ and
using complex coordinates
$z_j=x_j+\rmi y_j$, $j=1,\ldots,n-1-d$ on $\C^{n-1-d}$.
Again the tuple $(z_1,\ldots,z_{n-1-d})$
is abbreviated by $\bfz$
so that $\frac12|\bfy|^2$ reads as
\[
\frac12\sum_{j=1}^{n-1-d}y_j^2\;.
\]
In particular,
$(V,J_V)$ is foliated by holomorphic discs
$\{\bfw\}\times\D\times\{\bfs+\rmi\bft\}$.

%%%%%%%%%%%%%%%%%%%%%%%%%%%%%%%%%%%%%%%%%%%%%%%%%%%%%

\subsection{The symplectisation\label{subsec:symplectisation}}

Let $\tau\equiv\tau(a)$ be a
strictly increasing smooth function
$\R\ra(0,\infty)$.
We consider the symplectisation
\[
\big(\R\times Z, \rmd(\tau\alpha_0)\big)
\]
of $(Z,\alpha_0)$.
Define a compatible, translation invariant
almost complex structure $J$ 
that preserves the contact hyperplanes $\xi_0$
on all slices $\{a\}\times Z$
by requiring that $J(\partial_a)=\partial_b$
and that
\[
J\big(v-\lambda_V(v)\partial_b\big)
=
J_Vv-\lambda_V(J_Vv)\partial_b
\]
for all $v\in TV$.
The {\bf Niederkr\"uger map}
is the biholomorphism
\[
 \begin{array}{rccc}
  \Phi\co & (\R\times\R\times V,J)     & \lra & (\C\times V,\rmi\oplus J_V)\\
            & (a,b\,,\bfz) & \longmapsto     & \big(a-\psi(\bfz)+\rmi b,\bfz\big)
 \end{array}
\]
recalling that $Z=\R\times V$,
see \cite[Proposition 5]{nie06} and \cite[Proposition 2.1]{gz16b}.

%%%%%%%%%%%%%%%%%%%%%%%%%%%%%%%%%%%%%%%%%%%%%%%%%%%%%

\subsection{The Niederkr\"uger transform\label{subsec:niedtrans}}

Using the inverse of $\Phi$ we lift
the holomorphic discs
\[
\{a+\rmi b\}\times\{\bfw\}\times\D\times\{\bfs+\rmi\bft\}
\]
from $(\C\times V,\rmi\oplus J_V)$
to the symplectisation $(\R\times\R\times V,J)$
of $(Z,\alpha_0)$.
For fixed $b\in\R$, $\bfw\in T^*Q$,
and $\bfs,\bft\in\R^{n-1-d}$,
the resulting {\bf standard holomorphic discs}
\[
\D\lra
\R\times\R\times T^*Q\times\D\times\C^{n-1-d}
\]
are parametrised by
\[
u_{\bfs,b}^{\bft,\bfw}(z)=
\Big(
\tfrac14\big(|z|^2-1\big),b\,,\bfw,z,\bfs+\rmi\bft
\Big)\;,
\]
cf.\ \cite[Section 2.2]{gz16b}.

In order to set boundary conditions
for the standard discs we define
a $(n-1)$-dimensional family of cylinders
\[
L^{\bft}_q:=
\{0\}\times\R\times T^*_qQ\times\partial\D
\times\R^{n-1-d}\times\{\bft\}
\;,
\]
where $\bft\in\R^{n-1-d}$ and $q\in Q$
are the parameters. 
Observe,
that the $L^{\bft}_q$ foliate $\{0\}\times\partial Z$.
Furthermore
the restriction of $\rmd(\tau\alpha_0)$
to the tangent bundle of $\{0\}\times Z$ equals
$\tau(0)\rmd\alpha_0$,
which is a positive multiple of
\[
\rmd\lambda
+\rmd x_0\wedge\rmd y_0
+\rmd\bfx\wedge\rmd\bfy
\;.
\]
Therefore,
$L^{\bft}_q$ is a Lagrangian cylinder
because the dimension of $L^{\bft}_q$ is $n+1$.

%%%%%%%%%%%%%%%%%%%%%%%%%%%%%%%%%%%%%%%%%%%%%%%%%%%%%

\subsection{Class independence\label{subsec:classindy}}

Preparing the definition
of the moduli space $\WW$
we consider the space
$\R\times T^*_q Q\times\R^{n-1-d}$
of tuples $(b,\bfw,\bfs)$.
Assuming $n\geq2$ this space is at least $2$-dimensional,
so that the complement
of any ball in
$\R\times T^*_q Q\times\R^{n-1-d}$
is path-connected.
Therefore,
we find $R>0$ such that
\begin{enumerate}
  \item
  the shape $S$ is contained
  in the closed disc bundle
  $D_R\big(T^*Q\oplus\underline{\R}^{2n+1-2d}\big)$
  of radius $R$, and
  \item
  all standard discs
  $u_{\bfs,b}^{\bft,\bfw}$ of level $(q,\bft)$, $\bfw\in T^*_qQ$,
  that are contained in
  \[
    \R\times
      \Big(
        Z\setminus D_R\big(T^*Q\oplus\underline{\R}^{2n+1-2d}\big)
      \Big)
  \]
  are homotopic therein relative $L^{\bft}_q$ via a homotopy inside
  \[
    \{0\}\times\R\times
    T^*_qQ\times\D\times
    \R^{n-1-d}\times\{\bft\}
    \;.
  \]
\end{enumerate}

%%%%%%%%%%%%%%%%%%%%%%%%%%%%%%%%%%%%%%%%%%%%%%%%%%%%%

\section{Symplectic potentials on cotangent bundles\label{sec:potentialsont*q}}

We prepare the proof of geometric bounds
on holomorphic discs
that belong to the moduli space $\WW$.
The aim of this section is to
construct an almost complex structure
on $T^*Q$.

The almost complex structure
on $T^*Q$ that belongs to the Levi-Civita connection of $Q$
is the one that is induced by the kinetic energy function.
The one coming from symplectising
the unit cotangent bundle in contrast belongs to the length functional
and does not extend over the zero section.
Here we want to interpolate the two in order
to obtain $C^0$-bounds on holomorphic curves
in the complement of the unit codisc bundle
that we after all can identify
with the positive symplectisation also holomorphically.

%%%%%%%%%%%%%%%%%%%%%%%%%%%%%%%%%%%%%%%%%%%%%%%%%%%%%

\subsection{Dual connection\label{subsec:dualcon}}

We denote the covariant derivative
of the Levi-Civita connection of $Q$ by $\nabla$.
The corresponding covariant derivative $\nabla^*$
of the dual connection is defined via chain rule by
\[
\big(\nabla^*\beta\big)(X,Y)
:=\big(\nabla^*_X\beta\big)(Y)
:=X\big(\beta(Y)\big)-\beta(\nabla_X Y)
\]
for $1$-forms $\beta$ and
vector fields $X,Y$ on $Q$,
cf.\ \cite[Section 4]{bwz}.
Denoting the Christoffel symbols of $\nabla$
by $\Gamma_{ij}^k$
the Christoffel symbols $(\Gamma^*)_{ij}^k$ of $\nabla^*$
can be expressed by
$(\Gamma^*)_{ij}^k=-\Gamma_{ik}^j$.
The connection map
of the dual connection
$K\co TT^*Q\ra T^*Q$
and the tangent functor $T$ are related via
$K\circ T=\nabla^*$
and defines a splitting of
\[
TT^*Q=\HH\oplus\VV
\]
into horizontal
\[
\HH:=\ker\big(K\co TT^*Q\lra T^*Q\big)
\]
and vertical distribution
\[
\VV=\ker\big(T\tau\co TT^*Q\lra TQ)
\;,
\]
where $T\tau$ is the linearisation of the cotangent map
$\tau\co T^*Q\ra Q$.
Observe that 
$T\tau$ defines a bundle isomorphism from $\HH$ onto $\tau^*TQ$
and that $\VV$ can be identified with $\tau^*T^*Q$ canonically.

%%%%%%%%%%%%%%%%%%%%%%%%%%%%%%%%%%%%%%%%%%%%%%%%%%%%%

\subsection{Orthogonal splitting\label{subsec:orthosplitt}}

Denoting the metric of $Q$ by $g$,
contraction defines a bundle isomorphism
\[
 \begin{array}{rccc}
  G: & TQ  & \lra & T^*Q\\
            & v & \longmapsto  & i_vg\;.
 \end{array}
\]
The dual metric $g^{\flat}$
is defined by
\[
g^{\flat}(\alpha,\beta)=g\big(G^{-1}(\alpha),G^{-1}(\beta)\big)
\]
for co-vectors $\alpha,\beta\in T^*Q$ on $Q$,
so that the dual norm $\alpha\mapsto |\alpha|_{\flat}$
defines the {\bf length function} on $T^*Q$.
The {\bf kinetic energy function} reads as
\[
k(\beta)=\frac12|\beta|_{\flat}^2
\;.
\]
For a smooth,
strictly increasing function $\chi\co\R\ra\R$
with $\chi(0)=0$ we define
\[
F=\chi\circ k\co T^*Q\ra[0,\infty)
\;.
\]
This leads to a Riemannian metric $h$ on $T^*Q$ defined by
\[
h\big(v\oplus\alpha,w\oplus\beta\big)
:=
\frac{1}{\chi'\circ k}\cdot g\Big(T\tau(v),T\tau(w)\Big)+
(\chi'\circ k)\cdot g^{\flat}(\alpha,\beta)\;,
\]
where $v,w\in\HH$ and $\alpha,\beta\in\VV$. 
The metric $h$ turns $TT^*Q=\HH\oplus\VV$
into an orthogonal splitting.

%%%%%%%%%%%%%%%%%%%%%%%%%%%%%%%%%%%%%%%%%%%%%%%%%%%%%

\subsection{Taming structure\label{subsec:tamestruc}}

The Liouville form $\lambda$ on $T^*Q$
is given by
$\lambda_{\bfw}=\bfw\circ T\tau$ for $\bfw\in T^*Q$
and defines a symplectic form via $\rmd\lambda$.
Observe that
for $v,w\in\HH$ and $\alpha,\beta\in\VV$
\[
\lambda_u\big(v\oplus \alpha\big)
=\bfw\big(T\tau (v)\big)
\]
and
\[
\rmd\lambda\big(v\oplus\alpha,w\oplus\beta\big)
=
\alpha\big(T\tau(w)\big)-\beta\big(T\tau(v)\big)
\;.
\]
In view of the splitting
$TT^*Q=\HH\oplus\VV$
we define the almost complex structure
$J_{T^*Q}$ by setting
\[
J_{T^*Q}\big(v\oplus\alpha\big)
:=
(\chi'\circ k)\cdot G^{-1}(\alpha)
\oplus
\frac{-1}{\chi'\circ k}\cdot G(v)
\]
for $v\in\HH$ and $\alpha\in\VV$. 
This yields
\[
h=\rmd\lambda\big(\;.\;,J_{T^*Q}\;.\;\big)
\;,
\]
i.e.\ $J_{T^*Q}$ is compatible with the symplectic form $\rmd\lambda$.
Non-degeneracy of the metric $h$ and the symplectic form $\rmd\lambda$
shows that the almost complex structure $J_{T^*Q}$ is uniquely determined.

%%%%%%%%%%%%%%%%%%%%%%%%%%%%%%%%%%%%%%%%%%%%%%%%%%%%%

\subsection{Potentials\label{subsec:potentials}}

We claim that the function $F$
is a symplectic potential on
the tame symplectic manifold
$(T^*Q,\rmd\lambda,J_{T^*Q})$
in the sense that
\[
\lambda=
-\rmd F\circ J_{T^*Q}
\;.
\]
Indeed,
in local $(\bfq,\bfp)$-coordinates on $T^*Q$
induced by Riemann coordinates on $Q$ about
$\bfq\equiv\mathbf{0}$
we have
\[
\HH_{(\mathbf{0},\bfp)}=
\left\{
  \big(\mathbf{0},\bfp,\dot{\bfq},\mathbf{0}\big)
  \mid
  \dot{\bfq}\in \R^d
\right\}
\;,
\quad
\VV_{(\mathbf{0},\bfp)}=
\left\{
  \big(\mathbf{0},\bfp,\mathbf{0},\dot{\bfp}\big)
  \mid
  \dot{\bfp}\in \R^d
\right\}
\;,
\]
as well as
\[
\lambda_{(\mathbf{0},\bfp)}=\bfp\,\rmd\bfq
\;,
\quad
\rmd\lambda_{(\mathbf{0},\bfp)}=\rmd\bfp\wedge\rmd\bfq
\;,
\]
and
\[
\big(J_{T^*Q}\big)_{(\mathbf{0},\bfp)}=
\begin{pmatrix}
  0&{\chi'\big(\tfrac12p^ip^i\big)}\\
  -\Big(\chi'\big(\tfrac12p^ip^i\big)\Big)^{-1}&0
\end{pmatrix}
\]
using block matrix notation
and writing e.g.\
${\chi'\big(\tfrac12p^ip^i\big)}$
instead of
${\chi'\big(\tfrac12p^ip^i\big)}\mathbbm{1}$.
Because of
\[
\rmd F|_{(\mathbf{0},\bfp)}=
\chi'\big(\tfrac12p^ip^i\big)\cdot p^j\rmd p^j
\]
we get therefore
\[
-\rmd F\circ J_{T^*Q}|_{(\mathbf{0},\bfp)}
=p^j\rmd q^j|_{(\mathbf{0},\bfp)}
=\lambda_{(\mathbf{0},\bfp)}
\]
as claimed.

%%%%%%%%%%%%%%%%%%%%%%%%%%%%%%%%%%%%%%%%%%%%%%%%%%%%%

\subsection{Interpolating geodesic and normalised geodesic
flow\label{subsec:intgeonorgeoflow}}

We choose the strictly increasing function $\chi\co\R\ra\R$
from Section \ref{subsec:orthosplitt}
to satisfy $\chi(t)=t$ for $t\leq\tfrac14$
and $\chi(t)=\sqrt{2t}$ for $t\geq\tfrac12$
in order to interpolate the kinetic energy
with the length function.

We would like to understand
the interpolation given by $\chi$
in terms of symplectisation.
For that we consider the diffeomorphism
\[
 \begin{array}{rccc}
  \Phi:&\big(\R\times ST^*Q,\rme^a\alpha\big)&\lra&\big(T^*Q\setminus Q,\lambda\big)\\
            &(a,\bfw)&\longmapsto&\rme^a\bfw
 \end{array}
\]
of Liouville manifolds,
where $\alpha:=\lambda|_{TST^*Q}$.
Observe, that
\[
\Phi^*F(a,\bfw)=
\chi\circ k\big(\rme^a\bfw\big)=
\chi\big(\tfrac12\rme^{2a}\big)
\]
equals $\rme^a$ for $a\geq 0$.
Since $\Phi$ is a symplectomorphism
$I:=\Phi^*J_{T^*Q}$ is a compatible
almost complex structure
on the symplectisation
$\big(\R\times ST^*Q,\rmd(\rme^a\alpha)\big)$.
Moreover,
on the positive part 
$\{a>0\}$ of the symplectisation,
where $\Phi^*F=e^a$,
we obtain $\Phi^*\rmd F=\rme^a\rmd a$. 
Therefore,
\[
\rme^a\alpha=
\Phi^*\lambda=
\Phi^*\big(-\rmd F\circ J_{T^*Q}\big)=
-\rme^a\rmd a\circ I
\;,
\]
which implies
\[
\alpha=-\rmd a\circ I
\;.
\]
Consequently,
$I$ preserves the contact structure
$\xi=\ker\alpha\cap\ker(\rmd a)$
induced by $\alpha$ on all slices.
Moreover,
denoting the Reeb vector field of $\alpha$ by $R$
we get
\[
1=\alpha(R)=-\rmd a(IR)
\;.
\]
Hence,
\[
I\partial_a=R
\;.
\]
We remark that $\partial_a$
is the Liouville vector field of
$\big(\R\times ST^*Q,\rme^a\alpha\big)$.
Therefore,
$\Phi_*\partial_a=Y$,
where $Y$ is the Liouville vector field
on $T^*Q$ determined by $\lambda=i_Y\rmd\lambda$.

We claim that the almost complex structure $I$
is invariant under translation in $\R$-direction
along $\R^+\times ST^*Q$.
Indeed,
using local Riemann coordinates
as in Section \ref{subsec:potentials}
the restriction of $J_{T^*Q}$ to
$\{|\bfp|_{\flat}>1\}$
is given by
\[
\big(J_{T^*Q}\big)_{(\mathbf{0},\bfp)}
=
\begin{pmatrix}
0&\tfrac{1}{|\bfp|}\\
-|\bfp|&0
\end{pmatrix}
\]
abbreviating e.g.\ $|\bfp|=|\bfp|_{\flat}\mathbbm{1}$.
As the flow of $Y$ scales by $\rme^t$ in $\bfp$-direction
the pullback of $J_{T^*Q}$
with respect to the flow of $Y=\bfp\partial_{\bfp}$
at $(\mathbf{0},\bfp)$ equals
\[
\begin{pmatrix}
  \mathbbm{1}&0\\
  0&\rme^{-t}
\end{pmatrix}
\begin{pmatrix}
  0&\frac{\rme^{-t}}{|\bfp|}\\
  -\rme^t|\bfp|&0
\end{pmatrix}
\begin{pmatrix}
  \mathbbm{1}&0\\
  0&\rme^t
\end{pmatrix}
=
\begin{pmatrix}
  0&\frac{1}{|\bfp|}\\
  -|\bfp|&0
\end{pmatrix}
\;.
\]
This shows
that the Lie derivative
$L_YJ_{T^*Q}$ vanishes.
Hence,
$\Phi_*\partial_a=Y$ impies
$L_{\partial_a}I=0$,
i.e.\ $I_{(a,\bfp)}=I_{(a+t,\bfp)}$ for all $a,a+t>0$.

In other words,
$I$ is a compatible almost complex structure on
the positive part of the symplectisation
$\big(\R^+\times ST^*Q,\rmd(\rme^a\alpha)\big)$.
$I$ is translation invariant,
preserves the contact structure $\xi=\ker\alpha$,
and sends the Liouville vector field $\partial_a$
to the Reeb vector field $R$ of $\alpha$.

%%%%%%%%%%%%%%%%%%%%%%%%%%%%%%%%%%%%%%%%%%%%%%%%%%%%%

\section{A boundary value problem\label{sec:abdryvalprob}}

Following 
\cite[Section 3]{gz16b} and \cite[Section 3]{bschz19}
we introduce the moduli space $\WW$
of holomorphic discs
in order to understand the topology
of the manifold $M$.
We consider the glued
strict contact manifold $(\hat{Z},\hat{\alpha})$
introduced in Section \ref{subsec:comviagluing}
and form its symplectisation
$(W,\omega)$,
i.e.\ we set
\[
(W,\omega):=
\Big(
  \R\times\hat{Z},
  \rmd(\tau\hat{\alpha})
\Big)
\]
for a positive,
strictly increasing smooth function
$\tau$ defined on $\R$
such that $\tau(a)=\rme^a$ for all $a\geq0$.
Compared to the constructions in
\cite{gz16b,bschz19}
there will be a substantial difference
in setting up the boundary conditions
for the holomorphic discs.

%%%%%%%%%%%%%%%%%%%%%%%%%%%%%%%%%%%%%%%%%%%%%%%%%%%%%

\subsection{An almost complex structure\label{subsec:analmcpxstr}}

We denote by $\hat{\xi}$ the contact structure
defined by $\hat{\alpha}$.
On the symplectisation $(W,\omega)$
we choose a compatible
almost complex structure $J$ 
that is $\R$-invariant,
sends $\partial_a$ to the Reeb vector field of $\hat{\alpha}$,
and restricts to a complex bundle structure on
$\big(\hat{\xi},\rmd\hat{\alpha}\big)$.

In order to incorporate standard holomorphic discs
we define the {\bf box} $B$ by
\[
B
:=[-b_0,b_0]
\times D_RT^*Q
\times D^2_r
\times D^{2n-2-2d}_R
\,,
\]
where $0<b_0,r\in(0,1),1\leq R$ are real numbers
chosen such that $S\subset\Int B$.
Here
$D^{2\ell}_{\rho}\subset\C^{\ell}$ denotes
the closed $2\ell$-disc of radius $\rho$
and $D_{\rho}T^*Q$ is the closed $\rho$-disc
subbundle of $T^*Q$.
Set
\[
\hat{B}:=(B\setminus\Int D)\cup_{\varphi}M\;.
\]
We require the almost complex structure $J$
to be the one defined in Section \ref{sec:stholdiscs}
on the complement of $\R\times\Int(\hat{B})$
in $\R\times\hat{Z}$.
On $\R\times\Int(\hat{B})$
we will choose $J$ generically,
see Section \ref{sec:transversality}.

%%%%%%%%%%%%%%%%%%%%%%%%%%%%%%%%%%%%%%%%%%%%%%%%%%%%%

\subsection{The moduli space\label{subsec:themodspace}}

The {\bf moduli space} $\WW$
is the set of all holomorphic discs
\[
u=(a,f)\co\D\lra (W,J)
\]
that satisfy the following conditions:
\begin{enumerate}
\item [(w${}_1$)]
  There exists a {\bf level} $(q,\bft)\in Q\times\R^{n-1-d}$
  such that
  \[u(\partial\D)\subset L^{\bft}_{q}\,.\]
\item [(w${}_2$)]
  There exist $b\in\R$, $\bfw\in T_q^*Q$, $\bfs\in\R^{n-1-d}$
  such that
  \[
  [u]=[u_{\bfs,b}^{\bft,\bfw}]
  \in H_2(W,L^{\bft}_{q})\,,
  \]
  where $(q,\bft)$ is the level of $u$.
\item [(w${}_3$)]
  $u$ maps the marked points $1,\rmi,-1$
  to the characteristic leaves
  $L^{\bft}_{q}\cap\{z_0=1\}$,
  $L^{\bft}_{q}\cap\{z_0=\rmi\}$,
  and $L^{\bft}_{q}\cap\{z_0=-1\}$,
  resp., i.e.\ for $k=0,1,2$ we have
  \[
  f(\rmi^k)\in
  \R\times T^*_qQ\times\{\rmi^k\}
  \times\R^{n-1-d}\times\{\bft\}
  \,.
  \]
\end{enumerate}

  The parameters $b,\bfw,\bfs$
  in condition (w${}_2$)
  are assumed to be sufficiently large
  so that the standard disc
  $u_{\bfs,b}^{\bft,\bfw}$
  defines a holomorphic disc in $(W,J)$.
  With Section \ref{subsec:classindy}
  the relative homology class of $u_{\bfs,b}^{\bft,\bfw}$
  is independent of the choice of $b,\bfw,\bfs$.

%%%%%%%%%%%%%%%%%%%%%%%%%%%%%%%%%%%%%%%%%%%%%%%%%%%%%

\subsection{Uniform energy bounds\label{subsec:energybounds}}

The {\bf symplectic energy}
$\int_{\D}u^*\omega$ is bounded by $\pi$
for all $u=(a,f)\in\WW$.
Indeed,
by Stokes theorem,
the symplectic energy of $u$
is equal to the {\bf action} $\int_{\partial\D}f^*\hat{\alpha}$
of the boundary circle.
This also holds for any
standard disc homologous to $u$.
The claim follows
as the symplectic energy is the same for all
holomorphic discs of the same level
and as the action
of the boundary circle of standard discs equals $\pi$.

By a similar argument we obtain
that the symplectic energy
of any non-constant holomorphic disc
that takes boundary values in some
Lagrangian cylinder $L^{\bft}_{q}$
is a positive multiple of $\pi$.

%%%%%%%%%%%%%%%%%%%%%%%%%%%%%%%%%%%%%%%%%%%%%%%%%%%%%

\subsection{Maximum principle\label{subsec:maxprinc}}

Let $u=(a,f)\in\WW$ be a holomorphic disc
of level $(q,\bft)$.
By \cite[Lemma 3.6.(i)]{gz16b}
the function $a$ is subharmonic
and, hence, $a<0$ on $\Int\D$.
 
The set
$G:=f^{-1}\big(\hat{Z}\setminus\hat{B}\big)$
is an open subset of $\D$ that contains 
a neighbourhood of $\partial\D$ in $\D$.
Restricting $f$ to $G$
we can write
\[
f=\big(b,\bfw,h_0,\bfh\big)
\]
w.r.t.\ coordinate functions on
$\R\times T^*Q\times\D\times\C^{n-1-d}$.
As the Niederkr\"uger map is biholomophic
the function $b$ is harmonic
and the maps $\bfw,h_0,\bfh$ are holomorphic,
see Section \ref{subsec:symplectisation}.

In particular,
if $G=\D$,
then $u$ is one of the discs
$u_{\bfs,b}^{\bft,\bfw}$.
This follows as in \cite[Lemma 3.7]{gz16b}.
Simply use the fact that a holomorphic map
$\bfw\co\D\ra T^*Q$
with boundary on $T^*_qQ$ is constant
by Stokes theorem and
$\bfw^*\lambda=0$ on $\partial\D$.

Motivated by this
we introduce the notion of
standard holomorphic discs
to the glued manifold $W$:

\begin{defn}
\label{nonstan}
 A holomorphic disc
 $u=(a,f)\in\WW$ is a called a {\bf standard disc}
 if $f(\D)\subset\hat{Z}\setminus\Int\hat{B}$.
 Holomorphic discs $u=(a,f)\in\WW$ with
 $f(\D)\cap\Int\hat{B}\neq\emptyset$
 are called {\bf non-standard}.
\end{defn}

Applying the strong maximum principle
and the boundary lemma by E.\ Hopf to $h_0$
we obtain as in \cite[Lemma 3.6.(ii)]{gz16b}
and on \cite[p.~669 and p.~671]{gz16b}:
\begin{enumerate}
\item $f(\Int\D)\subset\Int\hat{Z}$.
\item $u|_{\partial\D}$ is an embedding.
\end{enumerate}

\begin{rem}
\label{rem:posimm}
 In the situation $u$ is a non-constant
 holomorphic disc $(W,J)$
 that satisfies just the boundary condition
 $u(\partial\D)\subset L^{\bft}_{q}$
 the conclusions from this section
 that rely on the maximum principle
 continue to hold.
 The corresponding replacement of the statement in (2)
 which does not use the homological assumption
 is the following:
 $h_0$ restricts to an immersion on $\partial\D$
 so that $u(\partial\D)$ is positively transverse
 to each of the characteristic leaves
 $L^{\bft}_{q}\cap\{z_0=\rme^{\rmi\theta}\}$,
 $\theta\in[0,2\pi)$. 
\end{rem}

\begin{rem}
\label{rem:monotonicity}
The monotonicity argument
used in \cite[Lemma 3.9]{gz16b} implies
that there exists a compact ball $K\subset\C^{n-1-d}$
such that $\bfh(G)\subset K$
for all non-standard disc $u\in\WW$,
i.e.\ with $u=(a,f)$ we have
 \[
 f^{-1}
 \Big(
   \R\times
   T^*Q\times
   \D\times
   \big(\C^{n-1-d}\setminus K\big)
 \Big)=\emptyset
 \;.
 \]
\end{rem}

%%%%%%%%%%%%%%%%%%%%%%%%%%%%%%%%%%%%%%%%%%%%%%%%%%%%%

\subsection{Integrated maximum principle\label{subsec:inmaxprinc}}

Let $u=(a,f)\in\WW$
be a holomorphic disc of level $(q,\bft)$.
As in Section \ref{subsec:maxprinc}
we consider
$G:=f^{-1}\big(\hat{Z}\setminus\hat{B}\big)$
so that we can write
$f=\big(b,\bfw,h_0,\bfh\big)$
on $G$.
In Section \ref{subsec:maxprinc}
we obtained uniform $C^0$-bounds
on $h_0$ and $\bfh$ relying on
the maximum principle from
\cite{gz16b,bschz19}.
As the boundary conditions in $T^*Q$-direction
are considerably different form the one used in
\cite{bschz19}
uniform $C^0$-bounds on $\bfw$
require a new argument.

First of all we remark
that by Stokes theorem
the symplectic energy of $u$
(which we computed
in Section \ref{subsec:energybounds}
to be equal to $\pi$)
is equal to the area
$\int_{\D}f^*\rmd\hat{\alpha}$
of $f$.
Because $f^*\rmd\hat{\alpha}$ is an area density
by our compatibility assumptions
we obtain
\[
\int_G\bfw^*\rmd\lambda
\leq
\int_{G}f^*\rmd\alpha_0
\leq
\pi\;.
\]
Recall the diffeomorphism
$\Phi\co\big(\R\times ST^*Q,\rme^a\alpha)\ra
\big(T^*Q\setminus Q,\lambda\big)$
of Liouville manifolds
from Section \ref{subsec:intgeonorgeoflow},
which pulls $J_{T^*Q}$ back to $I$.
Define $v:=\Phi^{-1}\circ\bfw$
and replace $G$ by the subset
$(|\bfw|)^{-1}\big((R,\infty)\big)$,
$R\geq1$ appearing in the definition
of the box in Section \ref{subsec:analmcpxstr},
so that
\[
v=(c,k)\co G\lra(\ln R,\infty)\times ST^*Q
\]
is an $I$-holomorphic map subject to the following boundary conditions:
\[
c\big(\partial G\setminus\partial\D\big)=\{\ln R\}\;,
\quad
k\big(\partial\D\cap G\big)\subset ST^*_qQ\;.
\]
Further we have
\[
\int_G
v^*\rmd(\rme^a\alpha)
\leq
\pi
\]
for the symplectic energy of $v$.

We consider the subdomain
\[
G_t:=c^{-1}\big((t, \infty)\big)
\]
of $G$ for $t\geq\ln R$.
Note that $G_{\ln R}=G$.
In order to allow partial integration
we denote by $\RR$ the set of all
regular values $t\in(\ln R,\infty)$
of the functions $c$
and $c|_{\partial\D\cap G}$.
By Sard's theorem $\RR$ has full measure. 
Therefore, the open set $\RR$
is dense in $(\ln R,\infty)$.

For $t\in\RR$ the domain $G_t$
has piecewise smooth boundary
\[
\partial G_t=
\partial\D\cap G_t+
\partial G_t\setminus\partial\D
\;,
\]
which we equip with the boundary orientation.
Up to a null set the interior boundary
$\partial G_t\setminus\partial\D$
is given by $c^{-1}(t)$.
Observe that $ST^*_qQ$
is a Legendrian sphere in the unit cotangent bundle
so that the restrictions of $k^*\alpha$
to the tangent spaces of $\partial\D\cap G_t$ vanish.
Stokes theorem applied twice implies
\[
\int_{G_t}v^*\rmd(\rme^a\alpha)
=\rme^t\int_{c^{-1}(t)}k^*\alpha
=\rme^t\int_{G_t}k^*\rmd\alpha
\;,
\]
where we used
$v^*\rmd(\rme^a\alpha)=\rmd(\rme^ck^*\alpha)$.

On the other hand
using Leibniz rule
we have a decomposition
\[
v^*\rmd(\rme^a\alpha)
=
\rme^c\rmd c\wedge k^*\alpha
+
\rme^ck^*\rmd\alpha
\]
into energy densities.
Define the $\alpha$-{\bf energy} functional by
\[
e(t):=
\int_{G_t}\rme^c\rmd c\wedge k^*\alpha
\geq0
\;.
\]
Therefore,
\[
\int_{G_t}v^*\rmd(\rme^a\alpha)=
e(t)+\int_{G_t}\rme^ck^*\rmd\alpha\geq 
e(t)+\rme^t\int_{G_t}k^*\rmd\alpha 
\]
using $\rme^c\geq\rme^t$ on $G_t$.

Combining these expressions for the symplectic energy
we get $e(t)\leq0$.
Hence, $e(t)=0$ for all $t\in\RR$,
i.e.\ the $\alpha$-energy functional $e=e(t)$
vanishes identically.
Because of
\[
\rmd c\wedge f^*\alpha=
\big(c_x^2 + c_y^2\big)\,\rmd x\wedge\rmd y
\]
we deduce that $c|_{G_t}=\rm{const}$
and, since $k^*\alpha=-\rmd c\circ\rmi$,
that $k^*\alpha|_{G_t}=0$
as well as $k^*\rmd\alpha|_{G_t}=0$.
We conclude that 
$v|_{G_t}=\rm{const}$
for all $t\in(\ln R,\infty)$.
An open and closed argument
for $G=(|\bfw|)^{-1}\big((R,\infty)\big)$
implies that either
$G=\emptyset$ or $v=\rm{const}$
on all of $G=\D$,
which in turn implies that $u\in\WW$ was a standard disc.
This shows uniform $C^0$-bounds in $T^*Q$-direction
for all non-standard discs $u\in\WW$:

\begin{prop}
\label{unic0bouint*qdrt}
 If $u=(a,f)\in\WW$ is a non-standard holomorphic discs,
 then
 \[
 f^{-1}
 \Big(
   \R\times
   \big(T^*Q\setminus D_RT^*Q\big)\times
   \D\times
   \C^{n-1-d}
 \Big)=\emptyset
 \;.
 \]
\end{prop}

%%%%%%%%%%%%%%%%%%%%%%%%%%%%%%%%%%%%%%%%%%%%%%%%%%%%%
%%%%%%%%%%%%%%%%%%%%%%%%%%%%%%%%%%%%%%%%%%%%%%%%%%%%%

\section{Compactness\label{sec:compactness}}

Consider a non-standard disc
$u=(a,f)\in\WW$ of level $(q,\bft)$.
On the preimage
$G:=f^{-1}\big(\hat{Z}\setminus\hat{B}\big)$
we write
\[
f=\big(b,\bfw,h_0,\bfh\big)
\]
w.r.t.\ to the decomposition
$\R\times T^*Q\times\D\times\C^{n-1-d}$.
In Section \ref{subsec:maxprinc} and
\ref{subsec:inmaxprinc}
we obtained uniform bounds on
\begin{enumerate}
  \item [(i)]
  $a$ from above by $0$,
  \item [(ii)]
  $h_0$ in the sense $|h_0|\leq1$,
  \item [(iii)]
  $\bfw$ and $\bfh$ in the sense that
  $|\bfw|_{\flat}$ and $|\bfh|$, resp.,
  are bounded by a geometric constant.
\end{enumerate}
The coordinate function $b$
completes to a holomorphic function
\[
a-
F(\bfw)-
\frac14|h_0|^2-
\frac12|\im\bfh\,|^2+
\rmi b
\]
on $G$,
where the restriction of the real part
to $\partial\D$ equals
$F(\bfw)|_{\partial\D}$
up to a constant. 
In \cite[Lemma 3.8]{gz16b},
where no $T^*Q$-component appears,
we used Schwarz reflection and
the maximum principle
to establish uniform bounds on $|b|$.
In our situation this would require real analyticity
of $F(\bfw)|_{\partial\D}$,
which in general does not hold.

We will work around this utilising a bubbling off analysis
that uses target rescaling along 
the Reeb vector field $\partial_b$ on
$\hat{Z}\setminus\hat{B}$.
This will require ideas from \cite{bwz}.
In fact,
by the elliptic nature of the holomorphic curves equation
the bubbling off analysis directly yields
compactness properties of holomorphic curves. 
Therefore,
we will combine the target rescaling in $b$-direction
with the usual target rescaling along 
the Liouville vector field $\partial_a$:

By the maximum principle
$|b|$ attains its maximum on $\partial G$.
Observe that because of
$f(\partial\D)\subset\hat{Z}\setminus\hat{B}$
the boundary of $G$ decomposes
\[
\partial G=\partial\D\sqcup f^{-1}(\partial\hat{B})
\;.
\]
Assuming $|b|\not\leq b_0$
we get therefore that $|b|$ attains its maximum
on $\partial\D$.

Suppose there exist sequences
$\zeta_{\nu} \in \D$ and 
$u_{\nu}=(a_{\nu},f_{\nu})\in\WW$
of non-standard such that
\[
|b_{\nu}(\zeta_{\nu})|\lra\infty
\]
writing $f_{\nu} =\big(b_{\nu},\bfw_{\nu},h_0^{\nu},\bfh_{\nu}\big)$.
We may assume that
$\zeta_{\nu}\in\partial\D$ for all $\nu$
and that $\zeta_{\nu}\ra\zeta_0$ in $\partial\D$.
By the mean value theorem
we find a sequence $z_{\nu}$ in $\D$
such that $|\nabla u_{\nu}(z_{\nu})|\ra\infty$.
This implies
that uniform gradient bounds
for non-standard holomorphic discs in $\WW$
result in uniform bounds on $b$.

\begin{prop}
\label{prop:bubboffinabdirc}
 Under the assumptions of Theorem \ref{thm:mainthm}
 each sequence of non-standard discs $u_\nu\in\WW$ 
 has a $C^\infty$-converging subsequence.
\end{prop}

\begin{proof}
Consider a sequence of non-standard discs
$u_{\nu}=(a_{\nu},f_{\nu})\in\WW$
of level $(q_{\nu},\bft_{\nu})$
such that $|\nabla u_{\nu}(z_{\nu})|\ra\infty$
for a sequence $z_{\nu}\ra z_0$ in $\D$.
By compactness of $Q$ and
Remark \ref{rem:monotonicity}
we can assume that
$(q_{\nu},\bft_{\nu})\ra(q_0,\bft_0)$.
Observe that modifications
as made in \cite[Section 4.1]{gz16b}
that fix the varying boundary conditions
we will mention in Section \ref{subsec:varboundconc}
are not necessary for the following compactness argument. 

Up to a choice of a subsequence
we distinguish two cases:
\begin{enumerate}
\item
 $f_{\nu}(z_{\nu})\in\hat{Z}\setminus\hat{B}$ for all $\nu$,
 and
\item
 $f_{\nu}(z_{\nu})\in\hat{B}$ for all $\nu$.
\end{enumerate}

In the first case, additionally,
we can assume that the sequences
$\bfw_{\nu}(z_{\nu})$, $h_0^{\nu}(z_{\nu})$,
and $\bfh_{\nu}(z_{\nu})$ converge
and that either
\begin{enumerate}
\item[(1.1)]
 $b_{\nu}(z_{\nu})\ra\pm\infty$, or
\item[(1.2)]
 $b_{\nu}(z_{\nu})\ra b_{\infty}\in\R$.
\end{enumerate}
In case (1.1) we use bubbling off analysis
as in \cite[Section 6]{gz13},
but this time applied to
the holomorphic maps
\[
\big(
  a_{\nu}-a_{\nu}(z_{\nu}),
  b_{\nu}-b_{\nu}(z_{\nu}),
  \bfw_{\nu},h_0^{\nu},\bfh_{\nu}
\big)
\]
defined on
$G_{\nu}:=f_{\nu}^{-1}\big(\hat{Z}\setminus\hat{B}\big)$
for interior bubbling;
for bubbling along the boundary
perform the shift w.r.t.\ the real parts $x_{\nu}$
of the $z_{\nu}$.
For both observe
that shift in $b$-direction
is a strict contactomorphism
of $(Z,\alpha_0)$ and
does not effect the Hofer energy.
In order to have enough space inside $G_{\nu}$
during the domain rescaling use the trick
in \cite[Case 1.2.b]{gz13}
explained on \cite[p.~547]{gz13};
this time make use of the stretching
of the holomorphic discs $u_{\nu}$
in $b$-direction instead of the $a$-direction.
In the cases (2) and (1.2)
apply the usual bubbling off analysis
as in \cite{fr08,fz15,hof93,hof99},
cf.\ \cite[Cases 1.1, 1.2.a, 2 in Section 6]{gz13}.

Finally, in all cases
we can argue as in \cite[Section 4]{gz16b}.
By the aperiodicity assumption $\inf_0(\alpha)\geq\pi$,
which with Section \ref{subsec:comviagluing}
implies $\inf_0(\hat{\alpha})\geq\pi$,
there is no bubbling off of finite energy planes.
This is because finite energy planes asymptotically converge
to contractible periodic Reeb orbits.
The asymptotic analysis of the finite energy planes
possibly requires a bubbling off analysis
that involves target rescaling in $b$-direction
as explained above,
cf.\ \cite[Section 5.2]{bwz}.

Because there are no bubble spheres
by exactness of $(W,\omega)$
we are left with bubbling off of holomorphic discs,
cf.\ \cite[Section 5.3]{bwz}.
This will lead us to a contradiction as in
\cite[Section 4.2]{gz16b}.
Indeed, the Hofer energy of a bubble discs
is a positive multiple of $\pi$,
see Section \ref{subsec:energybounds}.
As the Hofer energy of all $u_{\nu}$ equals $\pi$
by Section \ref{subsec:energybounds}
there is at most one bubble discs.
Hence, we can assume that $u_{\nu}$
converge in $C^{\infty}_{\loc}$
on $\D\setminus\{z_0\}$ for some $z_0\in\partial\D$.
By our assumption on the $3$ fixed marked points
in the definition of $\WW$
after removing the singularity $z_0$
the limiting holomorphic disc will be non-constant;
and, therefore,
will also have energy equal to a positive multiple of $\pi$.
But the sum of energies
of the bubble disc and the limiting disc
can not exceed $\pi$.
This contradiction shows uniform gradient bounds
for any sequence $u_{\nu}$
of holomorphic discs in $\WW$.
\end{proof}

%%%%%%%%%%%%%%%%%%%%%%%%%%%%%%%%%%%%%%%%%%%%%%%%%%%%%

\section{Transversality\label{sec:transversality}}

In Section \ref{sec:compactness}
we established properness
of the evaluation map
\[
 \begin{array}{rccc}
  \ev\co & \WW\times\D     & \lra & \hat{Z}\\
            & \bigl(u=(a,f),z\bigr) & \longmapsto     & f(z)\;.
 \end{array}
\]
The aim of this section is to show
that $\ev$ has degree $1$.
We will follow the considerations from
\cite[Section 5]{gz16b} and \cite[Section 3.5]{bschz19}
and just indicate the adaptations to the present situation.

%%%%%%%%%%%%%%%%%%%%%%%%%%%%%%%%%%%%%%%%%%%%%%%%%%%%%

\subsection{Maslov index\label{subsec:masovindex2}}

For all $u\in\WW$ the Maslov index
of the bundle pair
\[
\big(u^*TW,(u|_{\partial\D})^*TL^{\bft}_q\big)
\]
equals $2$,
where $(\bft,q)$ is the level of $u$.
Indeed,
following \cite[Lemma 3.1]{gz16b},
by homotopy invariance it suffices
to show the claim for standard discs
\[
u(z)=u_{\bfs,b}^{\bft,\bfw}(z)=
\Big(
\tfrac14\big(|z|^2-1\big),b\,,\bfw,z,\bfs+\rmi\bft
\Big)\;,
\]
$\bfw\in T_qQ$,
assuming
$W=\R\times\R\times T^*Q\times\D\times\C^{n-1-d}$.
In particular,
$u^*TW\cong\underline{\C}^{n+1}$.
Moreover,
$(u|_{\partial \D})^*TL^{\bft}_q$
is isomorphic to
$\rmi\R\oplus\rmi\R^d\oplus\rme^{{\rmi\theta}}\R\oplus\R^{n-1-d}$
over $\rme^{{\rmi\theta}}\in\partial\D$.
Hence, the Maslov index equals 2 by normalisation.

%%%%%%%%%%%%%%%%%%%%%%%%%%%%%%%%%%%%%%%%%%%%%%%%%%%%%

\subsection{Simplicity\label{subsec:simplicity}}

First of all we remark
that the classes $[u]\in H_2(W,L^{\bft}_q)$,
$u\in\WW$, are $J$-indecomposable.
Otherwise, we would find a decomposition
\[
[u]=\sum_{j=1}^N m_j[v_j]
\]
in $H_2(W,L^{\bft}_q)$,
for simple holomorphic discs $v_j$ with boundary on $L^{\bft}_q$
and multiplicities $m_j\geq 1$.
Writing $v_j=(a_j,f_j)$ we get for the energy
\[
\pi=\sum_{j=1}^N m_j\int_{\partial\D}f_j^*\alpha_0
\;.
\]
Writing
$\big(b_j,\bfw_j,h_0^j,\bfx_j+\rmi\bft_j\big)$
for the restriction of $f_j|_{\partial\D}$
the left hand side reads as
\[
\sum_{j=1}^N m_j\int_{\partial \D}
  \Big[
    b_j^*\rmd b+
    \bfw_j^*\lambda+
    (h_0^j)^*\tfrac12\big(x_0\rmd y_0-y_0\rmd x_0\big)-
    (\bfx_j+\rmi\bft_j)^*(\bfy\rmd\bfx)
  \Big]
  \;.
\]
The first and last summand vanish
by exactness of the form we pull back to the circle $\partial\D$;
the second vanishes because $\bfw_j(\partial\D)\subset T^*_qQ$.
Hence,
writing $r_j$ for the winding number of $h_0^j|_{\partial\D}$,
which is positive for non-constant $h_0^j$
by the argument priniciple,
we get
\[
\pi=
\pi\cdot\sum_{j=1}^N m_jr_j\geq
N\cdot\pi
\;.
\]
We conclude that
$N=1$, $m_1=1$,
i.e.\ $[u]$ is $J$-indecomposable.

Consulting \cite[Lemma 3.4]{gz16b}
we see that $u$ must be simple.
Because $u|_{\partial\D}$ is an embedding,
see Section \ref{subsec:maxprinc},
we obtain as in
\cite[Lemma 3.5]{gz16b}
that the set of $f$-injective points
is open and dense in $\D$.

%%%%%%%%%%%%%%%%%%%%%%%%%%%%%%%%%%%%%%%%%%%%%%%%%%%%%

\subsection{Variable boundary conditions\label{subsec:varboundconc}}

There is a natural way
to identify the boundary conditions
\[
L^{\bft}_q=
\{0\}\times\R\times T^*_qQ\times\partial\D
\times\R^{n-1-d}\times\{\bft\}
\]
for the holomorphic discs in $\WW$.
Observe,
that the union of $L^{\bft}_q$ over all
parameters $\bft\in\R^{n-1-d}$ and $q\in Q$
equals
\[
\{0\}\times\partial\hat{Z}=
\{0\}\times\R\times T^*Q\times\partial\D\times\C^{n-1-d}
\]
so that flows induced by tangent vectors
$\bfv\in T_{\bft}\R^{n-1-d}$ and
$v\in T_qQ$
can be taken for the identifications:
Consider a chart $(\R^d,0)\ra (Q,q)$
of $Q$ about $q$
and extend $v$ to a vector field on $\R^d$
that has compact support and is constant near $0$.
The induced flow on $Q$ naturally lifts to
a fibre and Liouville form preserving flow on $T^*Q$,
see \cite[p.~92]{mcsa98}.
Similarly,
extend $\bfv\in T_{\bft}\R^{n-1-d}$
to a compactly supported vector field
on $\R^{n-1-d}$ that is constant near $\bft\in\R^{n-1-d}$.

We regard $(v,\bfv)$
as a vector field on
$\R\times\R\times T^*Q\times\partial\D\times\C^{n-1-d}$
cutting off $(v,\bfv)$ with a bump function
that has support on a small neighbourhood
of $\{0\}\times [-b_0,b_0]\times T^*Q\times \partial \D\times \C$
and equals $1$ on a smaller neighbourhood.
We denote the corresponding flow on $W$
by $\psi_t^{(v,\bfv)}$.
Given a level $(q_0,\bft_0)$
we find a neighbourhood $U$
of $(q_0,\bft_0)\in Q\times\R^{n-1-d}$
and a vector field $(v,\bfv)$
as above
such that the time-$1$ map $\psi_1^{(v,\bfv)}$
sends $L^{\bft_0}_{q_0}$ to
$\psi_1^{(v,\bfv)}(L^{\bft_0}_{q_0})=L^{\bft}_q$
for all $(q,\bft)\in U$.
Simply define $(v,\bfv)$ to be
$(q-q_0,\bft-\bft_0)$ on $U$.

%%%%%%%%%%%%%%%%%%%%%%%%%%%%%%%%%%%%%%%%%%%%%%%%%%%%%

\subsection{Admissible functions\label{subsec:admissfunc}}

Denote by $\BB$ the separable Banach manifold
consisting of all continuous maps
$u:(\D,\partial\D)\ra\big(W,\{0\}\times\hat{C}\big)$ 
of Sobolev class $W^{1,p}$, $p>2$,
that satisfy the conditions (w${}_1$) - (w${}_3$)
in the definition of the moduli space $\WW$,
see Section \ref{subsec:themodspace}.
The Banach manifold structure is given as follows:
The subset $\BB_q^{\bft}\subset\BB$
of all $u$ of level $(q,\bft)$
is a separable Banach manifold
whose tangent spaces are
\[
T_u\BB_q^{\bft}=
W^{1,p}\big(u^*TW,(u|_{\partial\D})^*TL^{\bft}_q\big)
\;.
\]
Consider 
the level projection map
$\BB\ra Q\times\R^{n-1-d}$
that assigns to all $u\in\BB$
the corresponding level $(q,\bft)$.
Using the identifying maps
the $\psi_1^{(v,\bfv)}$
from Section \ref{subsec:varboundconc}
these defines a locally trivial fibration
on the Banach manifold $\BB$
with fibres $\BB_q^{\bft}$.

%%%%%%%%%%%%%%%%%%%%%%%%%%%%%%%%%%%%%%%%%%%%%%%%%%%%%

\subsection{Linearised Cauchy--Riemann operator \label{subsec:lincrop}}

In particular,
\[
T_u\BB=T_u\BB_q^{\bft}\oplus
\big(T_qQ\oplus\R^{n-1-d}\big)
\]
so that the linearised Cauchy--Riemann operator
at $u\in\BB$ of level $(q,\bft)$
splits as
\[
D_u=D_u^{(q,\bft)}\oplus K_u
\;,
\]
where $D_u^{(q,\bft)}:=D_u|_{T_u\BB_q^{\bft}}$
is the linearised Cauchy--Riemann operator
in fibre direction and
$K_u:T_qQ\oplus\R^{n-1-d}\ra L^p(u^*TW)$
is a compact perturbation,
see \cite[Section 5.1]{gz16b}.
The index of $D_u^{(q,\bft)}$ equals $n$,
as the Maslov index of the problem
with fixed boundary level was $2$
(see Section \ref{subsec:masovindex2}),
so that the total index equals $\ind D_u=2n-1$.

If $Q$ is oriented
we can orient $D_u$ via the determinant bundle
\[
\det D_u=
\det D_u^{(q,\bft)}\otimes
\det\big(T_qQ\oplus\R^{n-1-d}\big)
\]
as follows:
The line bundle $\det D_u^{(q,\bft)}$ is oriented
by the construction in \cite[Section 8.1]{fooo09}
via the trivial bundle
$TL^{\bft}_q\cong T^*_qQ\oplus\underline{\R}^{n+1-d}$
and the orientation of $T^*_qQ\cong\R^d$
so that the bundle pair
\[
\big(u^*TW,(u|_{\partial\D})^*TL^{\bft}_q\big)
\]
admits a natural trivialisation.
The line bundle
$\det\big(T_qQ\oplus\R^{n-1-d}\big)$
is oriented via the orientation of $Q\times\R^{n-1-d}$.

%%%%%%%%%%%%%%%%%%%%%%%%%%%%%%%%%%%%%%%%%%%%%%%%%%%%%

\subsection{Lifting topology\label{subsec:lifttop}}

As in \cite[Section 5.2]{gz16b}
we choose $J$ to be regular
by perturbing the induced complex structure
on $\hat{\xi}$ over $\hat{B}$.
Regularity of $J$ along standard discs is obvious.
Hence,
the moduli space 
$\WW$ is a smooth oriented manifold
of dimension $2n-1$
whose end is made out of
standard holomorphic discs.
Therefore,
the evaluation map $\ev$,
which is proper, has degree $1$.
With \cite[Section 6]{gz16b} and \cite[Section 2]{bgz19}
we see that $\ev$ induces
surjections of homology groups and of $\pi_1$.

Identify $Q$ with the subset
\[
Q\equiv\{0\}\times Q\times\{1\}\times\{0\}
\]
of
\[
\R\times T^*Q\times\{1\}\times\C^{n-1-d}
\subset\partial\hat{Z}
\;.
\]
Observe that $M$ is a strong deformation retract of $\hat{Z}$.
We choose a deformation retraction
such that the inclusion $Q\subset\hat{Z}$
is isotoped to an embedding $Q\ra M$.
Combining this with the following commutative diagram
\begin{diagram}
\WW \times \D & 
	\rTo^{\qquad\qquad\ev} & 
	\hat{Z}\\
\uTo_\subset &
	&
	\uTo_\subset\\
\WW\times\{1\} & 
	\rTo^{\quad\quad\ev\quad\quad} &
	\R\times T^*Q\times\{1\}\times\C^{n-1-d}\\
\end{diagram}
yields:

\begin{prop}
\label{prop:epimorphic}
 Under the assumptions of Theorem \ref{thm:mainthm}
 the isotoped inclusion $Q\ra M$
 induces a surjection of
 homology and fundamental groups.
\end{prop}

\begin{proof}
 This follows with the homology epimorphism argument
 from \cite[Section 2.3]{bgz19}
 and the covering argument from \cite[Section 2.5]{bgz19}.
\end{proof}

%%%%%%%%%%%%%%%%%%%%%%%%%%%%%%%%%%%%%%%%%%%%%%%%%%%%%%%%%%%%%%%%%%%%%%

\section{The homotopy type\label{sec:homotopytype}}

We compute the homotopy type of $M$ in terms of
$D\big(T^*Q\oplus\underline{\R}^{2n+1-2d}\big)$.
For that we assume that,
up to fibre preserving isotopy,
the shape $S$ is equal to the shape given by
the unit sphere bundle in
$T^*Q\oplus\underline{\R}^{2n+1-2d}$.
This results into the same construction for $\hat{Z}$
as in Section \ref{subsec:comviagluing}
up to ambient diffeotopy.

We identify $Q$ with the section of the sphere bundle
\[
\partial M=
S\big(T^*Q\oplus\underline{\R}^{2n+1-2d}\big)
\]
given by
\[
Q\equiv\{0\}\times Q\times\{1\}\times\{0\}
\]
in
\[
\R\times T^*Q\times\D\times\C^{n-1-d}
\;.
\]
Observe that this defines a natural embedding of
$D\big(T^*Q\oplus\underline{\R}^{2n+1-2d}\big)$
into $M$ via a small disc bundle about
\[
\{0\}\times Q\times\{(1-\varepsilon)\}\times\{0\}
\;,
\]
$\varepsilon>0$ small.
Indeed,
simply shift a small disc bundle in
$\R\times T^*Q\times\D\times\C^{n-1-d}$
in direction of
$\{0\}\times Q\times\{(1-\varepsilon)\}\times\{0\}$.
The image is denoted by $M_0$.

%%%%%%%%%%%%%%%%%%%%%%%%%%%%%%%%%%%%%%%%%%%%%%%%%%%%%

\subsection{Homology type and fundamental group}
\label{subsec:homotypeanfundgru}

Proposition \ref{prop:epimorphic} implies
that the inclusion $Q\subset M$
is surjective in homology and $\pi_1$.
Based on that we show:

\begin{prop}
 \label{prop:homologytyp}
 Under the assumptions of Theorem \ref{thm:mainthm}
 the inclusion $M_0\subset M$
 induces isomorphisms of homology groups.
\end{prop}

\begin{proof}
 The arguments are similar to
 \cite[p.~42]{bschz19} and \cite[Section 2.4]{bgz19}.
 Recall the general assumption $n-1\geq d$.
 
 From Proposition \ref{prop:epimorphic}
 we immediately obtain
 $H_kM=0$ for $k\geq d+1$ so that
 the homology isomorphism property
 of the inclusion $M_0\subset M$
 is automatic in all degrees $k\geq d+1$.
 
 By general position, 
 any section $Q\ra\partial M$
 of the sphere bundle induces an isomorphism
 in homology in degree $k\leq 2n-1-d$.
 Therefore,
 the inclusion of 
 the sphere bundle into the disc bundle of
 $T^*Q\oplus\underline{\R}^{2n+1-2d}$ is
 isomorphic in homology of degree $k\leq 2n-1-d$.
 We claim that the inclusion
 $\partial M\ra M$ shares the same property.
 With $d+1\leq 2n-1-d$ the proposition will be immediate.

 By Poincar\'e duality and
 the universal coefficient theorem
 we have
 \[
 H_k(M,\partial M)\cong
 H^{2n+1-k}M\cong
 FH_{2n+1-k}M\oplus TH_{2n-k}M
 \;,
 \]
 where $FH_*$ and $TH_*$
 denote the free and the torsion part of $H_*$, respectively.
 By the above $H_k(M,\partial M)=0$
 for $k\leq 2n-d-1$.
 The long exact sequence of the pair $(M,\partial M)$ implies
 that $\partial M\ra M$ is isomorphic in degree $k\leq 2n-2-d$
 and epimorphic in degree $k=2n-1-d$.
 Because the homology of the sphere bundle $\partial M$
 vanishes in degree $k=2n-1-d$
 the epimorphism is in fact injective.
\end{proof}

\begin{cor}
 \label{prop:pi1epi}
 Under the assumptions of Theorem \ref{thm:mainthm}
 the inclusion $M_0\subset M$
 induces an epimorphism on fundamental groups.
 If in addition $\pi_1Q$ is abelian,
 then the inclusion $M_0\subset M$
 will be $\pi_1$-isomorphic.
\end{cor}

\begin{proof}
 Using the $\pi_1$-isomorphism
 $M_0\simeq Q\subset\partial M$,
 the claim follows from
 Proposition \ref{prop:epimorphic}
 and \ref{prop:homologytyp}
 as in \cite[Section 2.5]{bgz19}.
\end{proof}

\begin{proof}[{\bf Proof of Theorem \ref{thm:mainthm} (i)}]
 The claim directly follows from
 Proposition \ref{prop:homologytyp}
 and Corollary \ref{prop:pi1epi}.
 Simply observe
 that the specific choice of section
 into the sphere bundle is irrelevant here.
\end{proof}

%%%%%%%%%%%%%%%%%%%%%%%%%%%%%%%%%%%%%%%%%%%%%%%%%%%%%

\subsection{A cobordism}
\label{subsec:acobordism}

Implementing the construction from
\cite[Section 4.2]{bschz19}
in the situation at hand
we define a cobordism
\[
X:=M\setminus\Int M_0
\;.
\]
The construction comes with the following diagram
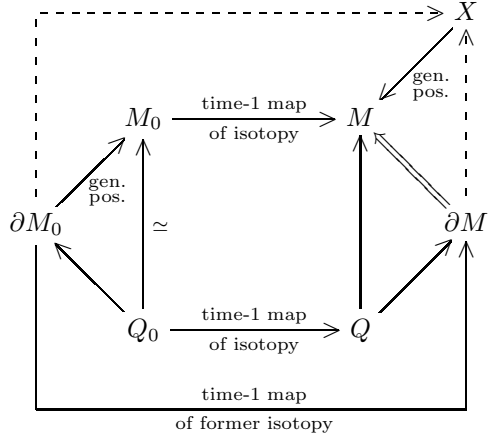
\begin{diagram}
\HmeetV  &
	\rDashto &
	&
	&
	&
	&
	X \\
&
	&
	&
	&
	&
	\ldTo_{\substack{\textrm{gen.}\\\textrm{pos.}}} &
	\\
&
	&
	M_0 &
	\rTo^{\textrm{time-1 map}}_{\textrm{of isotopy}} &
	M &
	&
	\\
\uDash&
	\ruTo_{\substack{\textrm{gen.}\\\textrm{pos.}}} &
	&
	&
	&
	\luDoubleto &
	\uDashto\\
\partial M_0 &
	&
	\uTo_{\simeq} &
	&
	&
	&
	\partial M \\
\dLine &
	\luTo &
	 &
	&
	\uTo &
	\ruTo&
	\\
&
	&
	Q_0 &
	\rTo^{\textrm{time-1 map}}_{\textrm{of isotopy}} &
	Q &
	&
	\uTo \\
\HmeetV &
	&
	&
	\rLine^{\substack{\\ \\ \\ \textrm{time-1 map}}}_{\textrm{of former isotopy}} &
	&
	&
	\HmeetV \\
\end{diagram}
that commutes up to homotopy.
We explain the diagram:
Set
\[
Q_0\equiv
\{0\}\times Q\times\{(1-\varepsilon')\}\times\{0\}
\;,
\]
where $\varepsilon'\in(0,\varepsilon)$ is chosen
such that $Q_0\subset\partial M_0$.
All arrows are given by inclusion
except those whose label refers to an isotopy.
The mentioned isotopy
is an isotopy of $Q_0$ inside $M$
that is the restriction
of a diffeotopy on
$\R\times T^*Q\times\D\times\C^{n-1-d}$
obtained by shifting and rescaling
that brings $Q_0$ to $Q$ and
$\partial M_0$ to $\partial M$.
The arrow $M_0\ra M$ is obtained
from an extension of the isotopy
of $Q_0\subset M$ to $M_0$.

\begin{prop}
 \label{prop:inclusionsboundary}
 Under the assumptions of Theorem \ref{thm:mainthm}
 the inclusion maps $\partial M_0,\partial M\subset X$
 induce isomorphisms of homology groups.
 If in addition $\pi_1Q$ is abelian
 (or more generally the inclusion $Q\subset M$
 is $\pi_1$-injective)
 then the inclusions $\partial M_0,\partial M\subset X$
 will be $\pi_1$-isomorphic.
\end{prop}

\begin{proof}
 The argumentation is the one
 given at the end of
 \cite[Section 4.2]{bschz19}:
 For low degrees $k\leq 2n-d-1$
 use general position arguments
 as indicated in the diagram
 and the results from
 Section \ref{subsec:homotypeanfundgru}.
 In higher degrees $k\geq d+1$
 essentially this is Poincar\'e duality
 and excision.
\end{proof}

\begin{proof}[{\bf Proof of Theorem \ref{thm:mainthm} part (a) in (ii) and (iii)}]
  We have to establish homotopy equivalence,
  resp., a diffeomorphism between $M$ and $M_0$.
  With Proposition \ref{prop:inclusionsboundary}
  this essentially follows from the relative Hurewicz
  and the $s$-cobordism theorem.
  The arguments are
  precisely as in the proof of \cite[Theorem 1.5]{bgz19}
  for $Q$ simply connected
  and \cite[Theorem 5.3]{bgz19}
  via finite coverings in the non-simply connected case.
\end{proof}

%%%%%%%%%%%%%%%%%%%%%%%%%%%%%%%%%%%%%%%%%%%%%%%%%%%%%

\subsection{Infinite coverings}
\label{subsec:infcov}

We assume the inclusion map $\partial M\subset M$
to be $\pi_1$-injective.
This will be satisfied
if $\pi_1Q$ is abelian for example.
If $Q$ is simply connected
vanishing in relative homology of the cobordism
$\{\partial M_0,X,\partial M\}$,
which will be simply connected too,
implies triviality of relative homotopy groups.
If $Q$ is not simply connected
one way to work around this
is to lift along the universal covering of $X$.
For $\pi_1Q$ finite
the universal covering space $\widetilde{X}$
will be compact
so that we are in the situation
of the previous sections.
This was used in the proof of
Theorem \ref{thm:mainthm} part (a) in (ii) and (iii)
in Section \ref{subsec:acobordism}.

If $\pi_1Q$ is infinite
we reset the moduli space:
The $\pi_1$-isomorphism $\partial M\subset M$ ensures
that the universal cover of $\hat{Z}$
is obtained by gluing similarly to Section \ref{subsec:comviagluing};
this time we glue the universal covers of the involved objects
along a lift of $\varphi$.
This makes it possible
to consider the moduli space $\WW'$
of holomorphic discs in $\widetilde{W}$
defined as in Section \ref{subsec:themodspace};
just replace $Q$ with $\widetilde{Q}$
in the definition of the Lagrangian boundary cylinders.
This places us into the situation of
\cite[Section 4.4]{bschz19}.
The change of the boundary condition is inessential
and the special choice $Q=T^d$ is not really used.
Hence, we obtain a covering $\WW'\ra\WW$
together with a proper degree $1$ evaluation map
\[
 \begin{array}{rccc}
  \ev\co & \WW'\times\D     & \lra & \widetilde{\hat{Z}}\\
            & \bigl(u=(a,f),z\bigr) & \longmapsto     & f(z)\;,
 \end{array}
\]
see \cite[Lemma 6.1]{bgz19}.
Similar to Proposition \ref{prop:epimorphic}
and \cite[Proposition 6.2 and Lemma 6.3]{bgz19} we obtain:

\begin{prop}
\label{prop:univcovepimorphic}
 Under the assumptions of Theorem \ref{thm:mainthm}
 the inclusion $\widetilde{Q}\ra\widetilde{M}$
 of universal covers
 induces a surjection of
 homology and fundamental groups.
 Further, the inclusion
 $\partial\widetilde{M}_0\ra\widetilde{X}$
 is homology surjective.
\end{prop}

Because the universal cover $\widetilde{X}$
is not compact for $\pi_1Q$ infinite
Poincar\'e duality delivers no information
about relative homology groups
in contrary to our argument in
Proposition \ref{prop:homologytyp}.
But we can say the following:

\begin{proof}[{\bf Theorem \ref{thm:mainthm} part (b) in (ii) and (iii)}]
 Because the universal cover of $Q$ is contractible
 so is $\widetilde{M}$ by
 Proposition \ref{prop:univcovepimorphic}.
 Hence, the inclusion
 $\widetilde{M}_0\subset\widetilde{M}$
 is a homotopy equivalence.
 This follows with the arguments from the proof of
 \cite[Theorem 7.2]{bgz19}.
 With the proof of \cite[Theorem 9.1]{bgz19},
 which in our situation
 is particularly easy because of the extra codimension,
 it follows that the boundary inclusions of
 $\widetilde{X}$ are homotopy equivalences.
 Hence, $X$ is in fact an $h$-cobordism.
 For the diffeomorphism type
 then apply the $s$-cobordism theorem.
\end{proof}

If $\partial M$ is a simple space,
which for example is satisfied whenever
$Q$ is a simple space and
$\partial M\ra Q$ a trivial sphere bundle,
then vanishing of relative homology of
$(X,\partial M_0)$ and $(X,\partial M)$, resp.,
implies homotopy equivalence
of each of the boundary inclusions of the cobordism
$\{\partial M_0,X,\partial M\}$.
The basic idea here is
that the kernel of the Hurewicz homomorphism
is made out of the action of the fundamental group,
which we now assume to be trivial,
see \cite[Section 8]{bgz19}:

\begin{proof}[{\bf Theorem \ref{thm:mainthm} part (c) in (ii) and (iii)}]
  Follows with the same arguments as in
  \cite[Theorem 1.7 and Example 9.3 (b)]{bgz19}.
\end{proof}

%%%%%%%%%%%%%%%%%%%%%%%%%%%%%%%%%%%%%%%%%%%%%%%%%%%%%%%%%%%%%%%%%%%%%%

\begin{ack}
  This work is inspired by the PhD Thesis \cite{kbdiss}
  of Kilian Barth and the Master Thesis \cite{jschdiss}
  of Jay Schneider who obtained the
  crucial $C^0$-bounds via a uniform relative monotonicity lemma
  in similar situations.
  We are grateful to Peter Albers, Hansj\"org Geiges, and Sefan Suhr
  for their valuable comments
  and to the University of Heidelberg for providing
  such a wonderful work environment
  that allowed us to finish this work. 
\end{ack}

%%%%%%%%%%%%%%%%%%%%%%%%%%%%%%%%%%%%%%%%%%%%%%%%%%%%%%%%%%%%%%%%%%%%%%

\end{document}